\documentclass[a4paper,12pt]{article}

\usepackage{amssymb}
\usepackage{latexsym}
\usepackage{amsfonts}
\usepackage{amsmath}
\usepackage{amsthm}
\usepackage{hyperref}
\usepackage{enumerate}
\usepackage{tensor}
\usepackage{mathtools}
\usepackage{mdwlist}
\usepackage{pgf,tikz,pgfplots}
\pgfplotsset{width=7.5cm,compat=1.9}

\usetikzlibrary{arrows,patterns,positioning}
\usetikzlibrary{calc}

\usepackage[margin=1in]{geometry}

\usepackage[compact]{titlesec}

\usepackage[scale=0.9]{tgheros}
\usepackage[T1]{fontenc}

\DeclareMathOperator{\Aut}{Aut}
\DeclareMathOperator{\Rep}{Rep}
\DeclareMathOperator{\Diag}{Diag}

\DeclareMathOperator{\PSL}{PSL}
\DeclareMathOperator{\SL}{SL}

\DeclareMathOperator{\PGU}{PGU}
\DeclareMathOperator{\pg}{PG}

\DeclareMathOperator{\Sp}{Sp}

\DeclareMathOperator{\Sz}{Sz}

\DeclareMathOperator{\N}{N}

\DeclareMathOperator{\AGaL}{A\Gamma L}
\DeclareMathOperator{\AGL}{AGL}
\DeclareMathOperator{\AG}{AG}
\DeclareMathOperator{\PG}{PG}

\DeclareMathOperator{\GL}{GL}

\DeclareMathOperator{\GaL}{\Gamma L}

\DeclareMathOperator{\Sym}{Sym}

\DeclareMathOperator{\supp}{supp}

\DeclareMathOperator{\alt}{A}

\renewcommand{\b}{\mathbf}
\renewcommand{\leq}{\leqslant}
\renewcommand{\geq}{\geqslant}

\newcommand{\Q}{\mathcal Q}

\newcommand{\Gtp}{{G_{{\b 0},i,j}}}
\newcommand{\QQ}{{\Q_i^\times\times \Q_j^\times}}
\newcommand{\Gtpta}{{G_{{\b 0},i,j}^{\QQ}}}
\newcommand{\Gp}{{G_{{\b 0},i}}}
\newcommand{\Gpi}{{G_{{\b 0},i}^{\Q_i^\times}}}
\newcommand{\Gtpi}{{G_{{\b 0},i,j}^{\Q_i^\times}}}
\newcommand{\Ki}{{K_{\b 0}^{\Q_i^\times}}}
\newcommand{\Gpkeri}{{G_{{\b 0},(\Q_i)}}}
\newcommand{\Gtpkeri}{{G_{{\b 0},(\Q_i),j}}}

\newcommand{\Gtpj}{G_{{\b 0},i,j}^{\Q_j^\times}}
\newcommand{\Gtpker}{G_{{\b 0},(\Q_i),(\Q_j)}}
\newcommand{\Kij}{K_{{\b 0}}^{\Q_i^\times\times \Q_j^\times}}

\newcommand{\C}{\mathcal C}
\newcommand{\F}{\mathbb F}

\newcommand{\D}{\mathcal D}

\newcommand{\nsub}{\vartriangleleft}

\renewcommand{\N}{\mathcal N}

\renewcommand{\S}{\mathcal S}
\newcommand{\RM}{{\mathcal{RM}}}
\newcommand{\PRM}{{\mathcal{PRM}}}

\newcommand{\Z}{\mathbb Z}

\theoremstyle{plain}
\newtheorem{lemma}{Lemma}
\newtheorem{theorem}[lemma]{Theorem}
\newtheorem{proposition}[lemma]{Proposition}

\theoremstyle{definition}
\newtheorem{definition}[lemma]{Definition}

\newtheorem{problem}[lemma]{Problem}
\newtheorem{remark}[lemma]{Remark}

\numberwithin{equation}{section}
\numberwithin{lemma}{section}

\begin{document}

\title{Alphabet-affine $2$-neighbour-transitive codes}
\author{Daniel R. Hawtin\footnote{The author has been supported by the Croatian Science Foundation under project number 4571.\\
Address: Faculty of Mathematics, University of Rijeka. Rijeka 51000. Croatia. Email: \href{mailto:dhawtin@math.uniri.hr}{dhawtin@math.uniri.hr}}}

\maketitle
\begin{abstract}
 A \emph{code} $\C$ is a subset of the vertex set of a Hamming graph $H(n,q)$, and $\C$ is \emph{$2$-neighbour-transitive} if the automorphism group $G=\Aut(\C)$ acts transitively on each of the sets $\C$, $\C_1$ and $\C_2$, where $\C_1$ and $\C_2$ are the (non-empty) sets of vertices that are distances $1$ and $2$, respectively, (but no closer) to some element of $\C$. 
 
 Suppose that $\C$ is a $2$-neighbour-transitive code with minimum distance at least $5$. For $q=2$, all `minimal' such $\C$ have been classified. Moreover, it has previously been shown that a subgroup of the automorphism group of the code induces an affine $2$-transitive group action on the alphabet of the Hamming graph. The main results of this paper are to show that this affine $2$-transitive group must be a subgroup of $\AGaL_1(q)$ and to provide a number of infinite families of examples of such codes. These examples are described via polynomial algebras related to representations of certain classical groups.
\end{abstract}

\section{Introduction}

There is a rich history of the study of symmetry in error-correcting codes in Hamming graphs, from perfect codes \cite{tietavainen1973nonexistence,zinoviev1973nonexistence}, to uniformly packed codes \cite{SemZinZai71} and completely regular codes \cite{delsarte1973algebraic}. For a survey on completely regular codes see \cite{borges2019completely}. Complete transitivity, the algebraic analogue of complete regularity, was introduced for binary linear codes in \cite{sole1990completely} and more generally in \cite{giudici1999completely}. This paper is concerned with $2$-neighbour-transitivity, a relaxation of complete transitivity; both concepts are defined formally below. To briefly introduce the notation, if $\C$ is a code in the Hamming graph $H(n,q)$ then $\C_i$ is the set of vertices of $H(n,q)$ that are distance $i$ from some element of $\C$, but not distance $j$ from any element of $\C$ for any $j<i$. Also, the \emph{covering radius} $\rho$ of $\C$ is the largest value of $i$ for which $\C_i$ is non-empty, the \emph{minimum distance} of $\C$ is the smallest distance between distinct elements of $\C$, and $\Aut(\C)$ is the automorphism group of $\C$ (see Section~\ref{sec:prelim}).

\begin{definition}\label{def:sneighbourtrans}
 Let $\C$ be a code with covering radius $\rho$ in the Hamming graph $\Gamma=H(n,q)$, let $G\leq \Aut(\C)$, and let $s\in\{1,\ldots,\rho\}$. Then we have the following definitions:
 \begin{enumerate}[(1)]
  \item $\C$ is \emph{$(G,s)$-neighbour-transitive} if $G$ acts transitively on each of the sets $\C,\C_1,\ldots, \C_s$.
  \item $\C$ is \emph{$G$-neighbour-transitive} if $\C$ is $(G,1)$-neighbour-transitive.
  \item $\C$ is \emph{$G$-completely transitive} if $\C$ is $(G,\rho)$-neighbour-transitive.
 \end{enumerate}
 Moreover, we say that $\C$ is \emph{neighbour-transitive}, \emph{$s$-neighbour-transitive}, or \emph{completely transitive}, respectively, if $\C$ is $\Aut(\C)$-neighbour-transitive, $(\Aut(\C),s)$-neighbour-transitive, or $\Aut(\C)$-completely transitive, respectively.
\end{definition}

Binary $2$-neighbour-transitive codes with minimum distance at least $5$ in $H(n,2)$ have been characterised via their minimal subcodes, see \cite{ef2nt,smallblocks,minimal2nt}. In particular, if $|\C|>2$ then either $\C$ is one of three non-linear codes, or $\C$ contains a known linear code, the possibilities for which are determined in \cite{minimal2nt}. These results have been used for a partial classification of binary completely transitive codes with minimum distance at least $5$ in $H(n,2)$; see \cite{bailey2023classification}. In the case where $q>2$, less is known. By \cite[Proposition~2.7]{ef2nt}, the automorphism group of $2$-neighbour-transitive code with minimum distance at least $5$ gives rise to a $2$-transitive action on the alphabet. It is known that every $2$-transitive group is either affine or almost-simple (see \cite[Section 154]{burnside1911theory} or \cite[Theorem 3.21]{PS2018}). In the case that the action on the alphabet is almost-simple, \cite[Theorem~1.1]{gillespie2017alphabet} proves that there are no $2$-neighbour-transitive codes with minimum distance at least $5$ in $H(n,q)$. This leads us to consider the case where the action on the alphabet is affine.

The first main result, stated below, provides information on the structure of the automorphism group of a $2$-neighbour-transitive code with minimum distance at least $5$ in $H(n,q)$. In particular, Theorem~\ref{thm:2ntCodes}(1) shows that the action on the alphabet for such a code is a subgroup of a $1$-dimensional affine semi-linear group. Here the vertices of $H(n,q)$ are $\N$-tuples over the alphabet $\Q$, with the set $\Q_i$ being the copy of $\Q$ in the $i$-th coordinate, where $i\in\N$. We assume that $0\in\Q$ and denote the set of non-zero elements of $\Q$ by $\Q^\times$. The group $G$ is a subgroup of the automorphism group of $H(n,q)$, and $G_{(\N)}$ is the subgroup of $G$ fixing $\N$ point-wise, that is, $G_{(\N)}$ is the kernel of the action of $G$ on $\N$. If $H\leq \Sym(\Q)$ then $\Diag_n(H)$ is the group consisting of all $n$-tuples $(h,\ldots,h)$, where $h\in H$, acting identically in each coordinate of the Hamming graph. We say that a code $\C$ is \emph{non-trivial} if $|\C|\notin \{0,1,q^n\}$, and a permutation group is \emph{semi-regular} if its point-stabiliser is trivial.

\begin{theorem}\label{thm:k0solubleandsemireg}
 Suppose $\C$ is a non-trivial $(G,2)$-neighbour-transitive code with minimum distance $\delta\geq 5$ in $H(n,q)=H(\N,\Q)$, let $K=G_{(\N)}$, let ${\b 0}\in\C$, and let $i\in \N$. Then the following hold:
 \begin{enumerate}[$(1)$]
  \item $G_i^{\Q_i}$ a $2$-transitive subgroup of $\AGaL_1(q)$.
  \item $K_{\b 0}\cong \Diag_n(H)$, where $H\leq\Sym(\Q)_0$ and $H$ acts semi-regularly on $\Q^\times$.
 \end{enumerate}
\end{theorem}

Note that Theorem~\ref{thm:2ntCodes}(1) implies that we may assume that the alphabet is $\F_q$ when discussing non-trivial $2$-neighbour-transitive codes with minimum distance at least $5$ in $H(n,q)$. However, this does not imply that such codes are linear; we discuss this further in the remark below. 

\begin{remark}\label{rem:nonLinear}
 Note that if $\C$ is a $G$-neighbour-transitive code with minimum distance $\delta\geq 3$ in $H(\N,\F_q)$, then \cite[Proposition~2.5]{ef2nt} implies that $G$ acts transitively on $\N$, and hence $G_i\cong G_j$ for all $i,j\in\N$. If we also assume that $\C$ is linear, that is, $\C$ is an $\F_q$-subspace of the vertex set of $H(\N,\F_q)$, then it follows that $\AGL_1(q)\leq G_i^{\Q_i}\leq \AGaL_1(q)$ and $K_{\b 0}^{\Q_i}=\Diag_n(\F_q^\times)$. There are several ways a code $\C$ satisfying Theorem~\ref{thm:k0solubleandsemireg} may be non-linear. First, $\C$ could be a union of cosets of a linear code, as is the case for the Nordstrom--Robinson codes of lengths $15$ and $16$, see \cite{semakov1969complete} or \cite{gillespie2017new}. It is also possible that $K_{\b 0}^{\Q_i}$ is a proper subgroup of $\GL_1(q)$, or that $K_{\b 0}^{\Q_i}$ is a subgroup of $\GaL_1(q)$ but is not contained in $\GL_1(q)$. An example of a regular group of the latter type is $\langle \omega^2,\omega\tau^2\rangle$, where $\omega$ is a generator of $\F_{3^4}^\times$ and $\tau$ is the Frobenius automorphism of $\F_{3^4}$ over $\F_3$ (see \cite[Lemmas~4.4 and~4.6]{li2009homogeneous}). The author is unaware of any examples of $2$-neighbour-transitive codes where $K_{\b 0}^{\Q_i}$ is strictly contained in $\GL_1(q)$, or where $K_{\b 0}^{\Q_i}$ is not a subgroup of $\GL_1(q)$.
\end{remark}

The next definition is required in order to state the second main result of the paper, Theorem~\ref{thm:2ntCodes}, which describes several infinite families of $2$-neighbour-transitive codes. Note that $N(a)=a^{(q^s-1)/(q-1)}$ is the norm of $a\in\F_{q^s}$, considered as an extension of $\F_q$. Also, the degree of a monomial $x_1^{a_1}\cdots x_t^{a_t}$ is $a_1+\cdots+a_t$, and the degree of a polynomial is the maximum value of the degrees of its constituent monomials. 

\begin{definition}\label{def:polynomials}
 Define $R(q,s,t,k)$ to be the set of all polynomials $f$ in $\F_{q^s}[x_1,\ldots,x_t]$ such that $f(a_1,\ldots,a_t)\in \F_q$ for all $(a_1,\ldots,a_t)\in\F_{q^s}^t$ and $f(ax_1,\ldots,ax_t)=N(a)^k f(x_1,\ldots,x_t)$, for each $a\in\F_{q^s}$. Note that the latter condition is equivalent to requiring that every monomial of $f$ has degree $k(q^s-1)/(q-1)$ modulo $q^s-1$.
\end{definition}

The relation of the above definition with codes in Hamming graphs comes from representing a vertex of $H(\N,\F_q)$ as a function $\N\rightarrow\F_q$. In particular, if $\N\subseteq\F_{q^s}^t$ and $f\in R(q,s,t,k)$ then $f$ represents a vertex of $H(\N,\F_q)$. Lemma~\ref{lem:polyHammingBij} treats this connection more concretely, and proves that $R(q,s,t,k)$ is an $\F_q \GL_t(q^s)$-module. The next theorem presents several infinite families of non-trivial $2$-neighbour-transitive codes; see Remark~\ref{rem:submoduleStr} for a discussion concerning the existence of the relevant $\F_q G_{\b 0}$-submodules. 

\noindent
\begin{table}[ht]
\begin{center}
 \begin{tabular}{cccc}
  $G_{\b 0}$ & $\N$ & $n$ & conditions\\
  \hline
  $\GL_t(q^s)/\Z_{(q^s-1)/(q-1)}$ & $\PG_{t-1}(q^s)$ & $\frac{q^{st}-1}{q^s-1}$ & $t\geq 2$\\
  $\F_q^\times\rtimes \AGL_{t-1}(q^s)$ & $\AG_{t-1}(q^s)$ & $q^{s(t-1)}$ & $t\geq 2$\\
  $\F_q^\times\rtimes \PGU_3(q^{s/2})$ & Classical unital & $q^{3s/2}+1$ & 
 $q=2^e$, $t=3$, $s$ even\\
  $\F_q^\times\rtimes \Sz(q^s)$ & Suzuki--Tits ovoid & $q^{2s}+1$ & 
 $q=2^{2e+1}$, $t=4$\\
  \hline
 \end{tabular}
 \caption{Some groups $G_{\b 0}$ acting transitively on the sets $\Gamma_1({\b 0})$ and $\Gamma_2({\b 0})$ of vertices in $H(\N,\F_q)$, where $\N$ is a set of representatives for the $1$-dimensional subspaces of the indicated subset of points of $\pg_{t-1}(q^s)$. See Theorem~\ref{thm:2ntCodes} and Section~\ref{sec:polyCodes} for more details.}
 \label{tab:2ntGroups}
\end{center}
\end{table}

\begin{theorem}\label{thm:2ntCodes}
 Let $q$, $s$, $t$, $G_{\b 0}$ and $\N$ be as in one of the lines of Table~\ref{tab:2ntGroups}, let $k\in\{1,2,\ldots,q-1\}$ with $\gcd(k,q-1)=1$, let $\C$ be an $\F_q G_{\b 0}$-submodule of $R(q,s,t,k)$ such that $\C$ is a nontrivial code with minimum distance $\delta$ and covering radius $\rho$ in $H(\N,\F_q)$, and let $G=T_{\C}\rtimes G_{\b 0}$. Then one of the following holds.
 \begin{enumerate}[$(1)$]
  \item $\rho=1$, $\C$ is $G$-completely transitive and either:
  \begin{enumerate}[$(i)$]
   \item $q=2$, $\delta=2$ and $\C$ is the dual of the binary repetition code, or,
   \item $\delta=3$ and $\C$ is a perfect Hamming code, that is, $G_{\b 0}$ and $\N$ are as in line 1 of Table~\ref{tab:2ntGroups}, with $s=1$ and $t=q-2$, and $\C$ is the submodule of $R(q,1,t,q-2)$ consisting of all polynomials having degree at most $\ell=(t-1)(q-1)-1$.
  \end{enumerate} 
  \item $\delta\geq 4$, $\rho\geq 2$ and $\C$ is $(G,2)$-neighbour-transitive.
 \end{enumerate}
\end{theorem}

\begin{remark}\label{rem:submoduleStr}
 Below we discuss some considerations concerning the existence of codes satisfying the above theorem, and their representation.
 \begin{enumerate}[(1)]
  \item Proposition~\ref{prop:2ntCodesMinDist} determines the minimum distances of certain codes under Theorem~\ref{thm:2ntCodes}(2) and lines 2--4 of Table~\ref{tab:2ntGroups}. In particular, this confirms that there are infinitely many $2$-neighbour-transitive codes arising in these cases.
  
  \item Note that, as per the present Definition~\ref{def:polynomials}, the space $R(q,s,t,k)$ is infinite-dimensional, unlike the vertex set of $H(\N,\F_q)$. Lemma~\ref{lem:polyHammingBij} remedies this by considering $\F_{q^s}[x_1,\ldots,x_t]/I$, where $I$ is the ideal generated by the set of all polynomials vanishing on $\bigcup_{v\in\N}\langle v\rangle_{\F_{q^s}}$. Indeed, this approach is fairly standard when studying polynomial-evaluation or algebraic-geometric codes; see, for example, \cite{gonzalez2024indicator,vladut1984linear}.
  
  \item If $\N$ is a set of representatives for the point-set of $\pg_{t-1}(q)$, then the $\F_q\GL_t(q)$-submodule structure of the image $A[k]$ of $R(q,1,t,k)$ in $\F_{q^s}[x_1,\ldots,x_t]/I$ (where $I$ is as in part (2) of this remark) is determined in \cite[Theorem~C]{sin2000permutationmodules}, where the parameter $k$ here is denoted $d$ in \cite{sin2000permutationmodules}. Briefly, if $q=p^r$ then composition factors of $A[k]$ are indexed by $r$-tuples of integers satisfying certain conditions. Furthermore, the lattice of ideals for a partial-order defined on the set of these $r$-tuples is shown to be isomorphic to the submodule lattice of $A[k]$. In this sense, all codes satisfying Theorem~\ref{thm:2ntCodes} with $s=1$, and $G_0$ and $\N$ as in line 1 of Table~\ref{tab:2ntGroups}, are known. See Section~\ref{sec:ReedMuller} for some discussion of the related generalised and projective Reed--Muller codes. 
  
  \item Further, new codes may be produced from the codes discussed in part (3) of this remark by restricting to a subset of $\N$, or by restricting the alphabet to a subfield. More formally, let $\C$ be an $\F_{q^s}\GL_t(q^s)$-submodule of $R(q^s,1,t,k)$, where $k=k'(q^s-1)/(q-1)$ and let $\C'$ be the $\F_{q}\GL_t(q^s)$-submodule consisting of those polynomials $f$ where $f(x_1,\ldots,x_t)\in\F_q$ for all $(x_1,\ldots,x_t)\in\F_{q^s}^t$. Then $\C'$ is an $\F_{q}\GL_t(q^s)$-submodule of $R(q,s,t,k')$ and, letting $s$, $t$, $\N$ be as in one of the lines of Table~\ref{tab:2ntGroups}, $\C'$ is a code in $H(\N,\F_q)$. That is to say, the submodule structure determined in \cite[Theorem~C]{sin2000permutationmodules} may be used to provide examples of codes satisfying Theorem~\ref{tab:2ntGroups}, for each line of Table~\ref{tab:2ntGroups}, for each prime power $q$, and for each integer $s\geq 1$. However, it is also worth noting that $R(q,s,t,k')$ may have a finer $\F_{q} G_{\b 0}$-submodule structure than that given by this process, and this has not been determined in general. 
 \end{enumerate}
\end{remark}

The partial classification results for binary completely transitive codes obtained in \cite{bailey2023classification} rely on knowledge of the maximal and second-maximal non-trivial $2$-neighbour-transitive codes in $H(n,q)$. In particular, it is often useful to know the minimum distance, the covering radius, and the geometry of the low-weight codewords of each code.

\begin{problem}
 Determine the maximal and second-maximal (by inclusion) non-trivial $2$-neighbour-transitive codes in $H(\N,\F_q)$ for each line of Table~\ref{tab:2ntGroups}, as well as bounds on their minimum distances and covering radii, and the geometry of the their low-weight codewords.
\end{problem}

Note that a significantly weaker version of Theorem~\ref{thm:k0solubleandsemireg}(1) was proved in the author's PhD thesis, as \cite[Theorem~8.1(2)]{myphdthesis}. Additionally, some parts of Theorem~\ref{thm:2ntCodes} were also proved in \cite{myphdthesis}; see \cite[Sections~9.1--9.5]{myphdthesis}. As far as the author is aware, the codes as in lines 2--4 of Table~\ref{tab:2ntGroups} under Theorem~\ref{thm:2ntCodes}(3) have not previously been studied.

The paper is organised as follows. The next section covers the notation and preliminary results required in later sections. The proof of Theorem~\ref{thm:k0solubleandsemireg} is developed in Section~\ref{sec:AlphabetAction}. The examples and proof of Theorem~\ref{thm:2ntCodes} are presented in Section~\ref{sec:polyCodes}, and certain Reed--Muller codes related to Theorem~\ref{thm:2ntCodes} are considered in Section~\ref{sec:ReedMuller}.

\section{Preliminaries}\label{sec:prelim}

Let $\N$ be a set of size $n$ and $\Q$ a set of size $q$, where $n,q\geq 2$. Depending on context, we will use the following two equivalent formulations of the Hamming graph, which we denote by $H(\N,\Q)$ or $H(n,q)$. In the first, we identify $\N$ with $\{1,\ldots,n\}$ and represent the vertex set of $H(\N,\Q)$ by the set of all $n$-tuples $(a_1,\ldots,a_n)$, where $a_i\in \Q$ for each $i\in \N$. An edge exists between two such $n$-tuples if and only if they differ in precisely one position. Alternatively, we may represent the vertex set of $H(\N,\Q)$ by the set of all functions $\alpha:\N\rightarrow\Q$, in which case an edge exists between functions $\alpha$ and $\beta$ if and only if there exists a unique $i\in\N$ such that $\alpha(i)\neq\beta(i)$. We refer to the set $\N$ as the \emph{set of entries} or the \emph{coordinates}, and the set $\Q$ as the \emph{alphabet}, of $H(\N,\Q)$. If $0\in\Q$ then the \emph{support} of a vertex $\alpha$ of $H(\N,\Q)$, denoted $\supp(\alpha)$, is the set $\{i\in\N\mid\alpha(i)\neq 0\}$, and the \emph{weight} of $\alpha$ is the size of $\supp(\alpha)$.

Let $\C$ be a code in $H(\N,\Q)$. As stated in the introduction, if $|\C|=0,1$ or $q^n$ then we say that $\C$ is trivial and we generally assume that $\C$ is non-trivial, sometimes without statement. Recall that the elements of $\C$ are called codewords and the minimum distance $\delta$ of $\C$ is the smallest distance in $H(\N,\Q)$ between a pair of distinct codewords, and the covering radius of $\C$ is the largest distance from any vertex of $H(\N,\Q)$ to its nearest codeword. The \emph{error-correction capacity} of $\C$, denoted $e$, is $\lfloor (\delta-1)/2\rfloor$. 

\subsection{Automorphism groups}

Let $G$ be a group acting on a set $\Omega$. Then we write $G^\Omega$ for the (faithful) subgroup of $\Sym(\Omega)$ induced by $G$. If $\alpha\in\Omega$ and $\Delta\subseteq\Omega$, then we write $G_\alpha$ for the stabiliser in $G$ of $\alpha$, we write $G_\Delta$ for the set-wise stabiliser in $G$ of $\Delta$, and we write $G_{(\Delta)}$ for the point-wise stabiliser in $G$ of $\Delta$. Thus $G_{(\Delta)}$ is the kernel for the action of $G_\Delta$ on $\Delta$, and we have that $G^\Omega=G/G_{(\Omega)}$. See \cite{dixon1996permutation} for more on permutation groups.

The full automorphism group of the Hamming graph $\Gamma=H(\N,\Q)$ factorises as the semi-direct product $\Aut(\Gamma)=\Aut(\Gamma)_{(\N)}\rtimes \Aut(\Gamma)^\N$, where $\Aut(\Gamma)_{(\N)}$ is isomorphic to $\Sym(\Q)^n$ and is called the \emph{base group}, and $\Aut(\Gamma)^\N$ is isomorphic to $\Sym(\N)$ and is called the \emph{top group}; see \cite[Theorem 9.2.1]{brouwer}. Let $x=h\sigma\in\Aut(\Gamma)$, where $h=(h_1,\ldots,h_n)\in \Sym(\Q)^n$ and $\sigma\in \Sym(\N)$. If $i'=i^{\sigma^{-1}}$ for each $i\in \N$, then $h, \sigma$ and $x$ act on an $n$-tuple $\alpha=(a_1,\ldots,a_n)$ via
\begin{equation}\label{e:hamaut}
 \alpha^h= \left(a_1^{h_1},\ldots,a_n^{h_n}\right), \quad  
 \alpha^\sigma= \left(a_{1'},\ldots,a_{n'}\right), \quad 
 \text{and}\quad \alpha^x=\left(a_{1'}^{h_{1'}},\ldots,a_{n'}^{h_{n'}}\right).
\end{equation}
For example, $(a_1,a_2,a_3,a_4)^{(1\,2\,3\,)}=(a_3,a_1,a_2,a_4)$. If instead we consider a vertex $\alpha$ to be a function $\N\rightarrow\Q$, then $h, \sigma$ and $x$ act on $\alpha$ via
\begin{equation}\label{e:hamaut2}
\alpha^h(i)=\left(\alpha(i)\right)^{h_i},\quad \alpha^\sigma(i)=\alpha(i'),  \quad \text{and}\quad  \alpha^x(i)=\left( \alpha(i')\right)^{h_{i'}},
\end{equation}
where $i'=i^{\sigma^{-1}}$ for each $i\in \N$. Let $G\leq \Aut(\Gamma)$. If $x=h\sigma\in G$, with $h=(h_1,\ldots,h_n)\in \Sym(\Q)^n$ and $\sigma\in \Sym(\N)$, then the image $G^\N$ of the homomorphism $x\mapsto x^\N=\sigma$ is the \emph{action (of $G$) on entries}, and the image $G_i^{\Q_i}$ of the homomorphism $x\mapsto x^{\Q_i}=h_i$ is the \emph{action (of $G$) on the alphabet in entry $i$}, where in the latter homomorphism we have assumed that $i^\sigma=i$. If $G_i^{\Q_i}\cong G_j^{\Q_j}$ for all $i,j\in\N$ then we simply refer to the \emph{action on the alphabet}. We will often denote the kernel $G_{(\N)}$ of the action of $G$ on $\N$ by $K$. If $\C$ is linear then we denote by $T_\C$ the group of translations by elements of $\C$.

\subsection{\texorpdfstring{$s$}{s}-Neighbour-transitive codes}

The following two propositions are fundamental results in the analysis of $s$-neighbour-transitive codes.

\begin{proposition}\label{prop:HammingiHomogeneous}\cite[Proposition~2.5]{ef2nt}
 Let $\C$ be a $(G,s)$-neighbour-transitive code with error-correction capacity $e\geq 1$ in $H(\N,\Q)$. If $\alpha\in \C$, then $G_\alpha$ acts $i$-homogeneously on $\N$, for each $i\leq\min\{e,s\}$.
\end{proposition}

\begin{proposition}\label{prop:Hamming2transOnQ}\cite[Proposition~2.7]{ef2nt}
 Let $\C$ be a $G$-neighbour-transitive code with minimum distance $\delta\geq 3$ in $H(\N,\Q)$. If $i\in \N$, then $G_i^{\Q_i}$ acts $2$-transitively on $\Q_i$.
\end{proposition}

By an old theorem of Burnside (\cite[Section 154]{burnside1911theory}, or see \cite[Theorem 3.21]{PS2018}) every finite $2$-transitive group is either a group of affine transformations of a finite vector space, or is an almost-simple group. Thus, Proposition~\ref{prop:Hamming2transOnQ} implies that every $(G,2)$-neighbour-transitive code satisfies precisely one of the conditions in Definition~\ref{efaasaadef}, below.

\begin{definition}\label{efaasaadef}
 Let $\C$ be a $G$-neighbour-transitive code in $H(\N,\Q)$, let $K$ be the kernel of the action of $G$ on $\N$, let $i\in \N$, and let $\Q_i$ be the copy of the alphabet corresponding to the $i$-th entry. Then precisely one of the following holds for $(\C, G)$; i.e.  $\C$ is 
 \begin{enumerate}[(1)]
  \item \emph{$G$-entry-faithful} if $G$ acts faithfully on $\N$, that is, $K=1$;
  \item \emph{$G$-alphabet-almost-simple} if $K\neq 1$, $G$ acts transitively on $\N$, and $G_i^{\Q_i}$ is a $2$-transitive almost-simple group; and
  \item \emph{$G$-alphabet-affine} if $K\neq 1$, $G$ acts transitively on $\N$, and $G_i^{\Q_i}$ is a $2$-transitive affine group.
 \end{enumerate}
\end{definition}

Let $\C$ be a $(G,2)$-neighbour-transitive code with minimum distance at least $5$ in $H(\N,\Q)$. If $\C$ is $G$-entry-faithful, then $\C$ is classified in \cite[Theorem~1.1]{ef2nt}. Moreover, by \cite{gillespie2017alphabet}, $\C$ is not $G$-alphabet-almost-simple. The next proposition concerns $G$-alphabet-affine codes. Note that $K$ is the point-wise stabiliser in $G$ of $\N$, and $O_p(K)$ is the largest normal $p$-subgroup of $K$.

\begin{proposition}\cite[Proposition~3.5]{minimal2nt}\label{orbitof0is2ntmodule}
  Let $\C$ be a  code in the Hamming graph $H(n,q)$, with $q=p^d$ for a prime $p$, such that $\C$ is $G$-alphabet-affine and $(G,2)$-neighbour-transitive, with $\delta\geq 5$, and suppose that $\b 0\in C$. Then $\C$ contains a subcode $\S$ such that $\S$ is the code formed by the orbit of $\b 0$ under $O_p(K)$, where $K=G\cap B$. Moreover, it follows that: 
 \begin{enumerate}[$(1)$]
  \item $\S$ is a block of imprimitivity for the action of $G$ on $\C$, and $G_{\S}=O_p(K)\rtimes G_{\b 0}$,
  \item $\S$ is $G_{\S}$-alphabet-affine and $(G_{\S},2)$-neighbour-transitive with minimum distance $\delta_{\S}\geq \delta$,
  \item $\S$ is an $\F_p G_{\b 0}$-module, and if $\S\neq \Rep_n(2)$ then $q^2$ divides $|\S|$.
 \end{enumerate}
\end{proposition}

The concept of a \emph{$q$-ary design} is defined below, and a connection to $s$-neighbour-transitive codes is given in Lemma~\ref{lem:design}. Let $\alpha\in H(\N,\Q)$ and $0\in \Q$. The vertex $\nu$ is said to be \emph{covered} by $\alpha$, if for every $i\in \N$ such that $\nu_i\neq 0$ we have $\nu_i=\alpha_i$. 

\begin{definition}\label{desdef}
 A \emph{$q$-ary $t$-$(v,k,\lambda)$ design} is a subset $\mathcal{D}$ of vertices of $\Gamma_k(\b 0)$ (where $k\geq t$) such that each vertex $\nu \in\Gamma_t(\b 0)$ is covered by exactly $\lambda$ vertices of $\mathcal{D}$. When $q=2$, $\D$ is simply the set of characteristic vectors of a $t$-design. We refer to the elements of $\D$ as \emph{blocks}.
\end{definition}



\begin{lemma}\label{lem:design}\cite[Lemma~2.16]{ef2nt}
 Let $\C$ be a $(G,s)$-neighbour-transitive code. Then $\C$ is $s$-regular. Furthermore, if $\b 0\in \C$ and $\delta\geq 2s$ then the set of codewords of weight $k\leq n$ forms a $q$-ary $s$-$(n,k,\lambda)$ design, for some $\lambda$.
\end{lemma}





The next result determines properties of a code $\C$ in $H(\N,\Q)$ given certain conditions on the local action of a group its automorphism group.

\begin{proposition}\label{prop:coveringRadius1}
 Let $\C$ be a non-trivial code with covering radius $\rho$ and minimum distance $\delta$ in $\Gamma=H(\N,\Q)$, let $\alpha\in\C$ and let $G\leq\Aut(\C)$ where $G$ acts transitively on $\C$ and $G_\alpha$ acts transitively on the sets $\Gamma_1(\alpha)$ and $\Gamma_2(\alpha)$. One of the following holds:
 \begin{enumerate}[$(1)$]
  \item $\rho\geq 2$, $\delta\geq 4$ and $\C$ is $(G,2)$-neighbour-transitive.
  \item $q=2$, $\rho=1$, $\C$ is $G$-neighbour-transitive, but not $(G,2)$-neighbour-transitive, and one of the following holds:
  \begin{enumerate}[$(i)$]
   \item $\delta=3$ and $\C$ is perfect.
   \item $\delta=2$ and $\C$ is equivalent to the dual of the binary repetition code.
  \end{enumerate}
 \end{enumerate}
\end{proposition}

\begin{proof}
 Without loss of generality, we assume that $\alpha={\b 0}\in\C$ and $G_{\b 0}$ acts transitively on $\Gamma_i({\b 0})$ for $i=1,2$. If $\rho=0$ then $\C=V(\Gamma)$ is a trivial code, but since this is not the case we have $\rho\geq 1$. 
 
 Suppose $\rho\geq 2$. Since $G_{\b 0}$ is transitive on $\Gamma_1({\b 0})$ and $\Gamma_2({\b 0})$, which are contained in $\C_1$ and $\C_2$, respectively, it follows that there are no weight $1$ or $2$ vertices in $\C$. If there was a weight $3$ vertex contained in $\C$, then any weight $2$ vertex adjacent to it would be in $\C_1$. Since this is not the case, and since $G$ is transitive on $\C$, it follows that $\delta\geq 4$. Thus $\C_1=\bigcup_{\beta\in\C}\Gamma_1(\beta)$ and $\C_2=\bigcup_{\beta\in\C}\Gamma_2(\beta)$. The fact that $G$ acts transitively on $\C$ and $G_{\b 0}$ acts transitively on $\Gamma_1({\b 0})$ and $\Gamma_2({\b 0})$ then implies that $\C$ is $(G,2)$-neighbour-transitive, and part (1) holds. Hence, we may assume that $\rho=1$. 
 
 Since $G$ acts transitively on $\C$ and $G_{\b 0}$ acts transitively on $\Gamma_1({\b 0})$, it follows that 
 \[
  \C_1=\bigcup_{\beta\in\C}\Gamma_1(\beta).
 \]
 In particular $\delta\geq 2$. If $\delta\geq 3$, then this union is disjoint, and since $\rho=1, |\C|\geq 2$ and $\Gamma$ is connected, there must be an edge between some vertex of $\Gamma_1(\beta)$ and some vertex of $\Gamma_1(\beta')$ for some distinct codewords $\beta$ and $\beta'$, and hence $d(\beta, \beta')=3$, so $\delta =3$. This implies that $\C$ has error-correction capacity $e=1$, and any pair of balls of radius $1$ centered at distinct codewords are disjoint. Moreover, since $\rho=1$, the vertex set $V(\Gamma)=\C\cup \C_1$, and hence the set of balls of radius $1$ centered at the codewords of $\C$ partitions $V(\Gamma)$. Thus $\C$ is perfect, as in part (2)(i).
 
 Thus we may assume that $\delta=2$. Then, since $G$ acts transitively on $\C$ and $G_{\b 0}$ acts transitively on $\Gamma_2({\b 0})$, it follows that $\Gamma_2({\b 0})$ is contained in $\C$. If $a,b\in\Q\setminus\{0\}$, with $a\neq b$, then the distance between $(a,a,0,\ldots,0)$ and $(a,b,0,\ldots,0)$ is $1$. Since $(a,a,0,\ldots,0),(a,b,0,\ldots,0)\in\Gamma_2({\b 0})\subseteq\C$ and $\delta=2$ we deduce that $q=2$. Let $\beta\in\Gamma_1({\b 0})$. Then $\beta$ is not adjacent to any weight $1$ vertex, so that $\Gamma_1(\beta)\subseteq\Gamma_2({\b 0})\cup\{{\b 0}\}\subseteq\C$. Hence, $\beta$ has no neighbours in $\C_1$. Since $G_{\b 0}$ is transitive on $\Gamma_1({\b 0})$, $G$ is transitive on $\C$ and $V(\Gamma)=\C\cup \C_1$, it follows that $\C$ and $\C_1$ are the bipartite halves of $H(n,2)$. Since the set of all even-weight vertices is precisely the dual of the binary repetition code, part (2)(ii) holds. 
\end{proof}

\begin{remark}\label{rem:perfectHamming}
 Note that if we add the assumption that $\C$ is linear to Proposition~\ref{prop:coveringRadius1} then part (2)(i) can be strengthened to say that $\C$ is a perfect Hamming code, as follows. By \cite[Theorem~37, Chapter~6] {macwilliams1978theory} a perfect linear code $\C$ with covering radius $1$ in $H(n,\F_q)$ necessarily has length $n=(q^k-1)/(q-1)$, dimension $k$ and minimum distance $3$. The condition `minimum distance $3$' implies that each column of a parity-check matrix $H$ for $\C$ is non-zero, and no pair of columns of $H$ is linearly dependent. This implies that the columns of $H$ are a set of representatives for the $1$-dimensional subspaces of $\F_q^k$, {\em i.e.}, that $\C$ is a perfect Hamming code.
\end{remark}

\section{Action on the alphabet}\label{sec:AlphabetAction}

The aim of this section is to prove Theorem~\ref{thm:k0solubleandsemireg}. First, we analyse the stabiliser of the zero codeword inside the kernel of the action of entries for a group of automorphisms of a $2$-neighbour-transitive code with minimum distance at least $5$ in $H(n,q)$. Recall that if $H\leq\Sym(\Q)$ then $\Diag_n(H)$ is the subgroup $\{(h,\ldots,h)\mid h\in H\}$ of $\Aut(H(\N,\Q))_{(\N)}$.

\begin{lemma}\label{lem:semireg}
 Let $\C$ be a $(G,2)$-neighbour-transitive code with minimum distance $\delta\geq 5$ in the Hamming graph $H(\N,\Q)$, let $K$ be the kernel of the action of $G$ on $\N$, and let ${\b 0}\in\C$. Then $K_{\b 0}\cong \Diag_n(H)$ where $H$ acts semi-regularly on $\Q_i^\times$ for all $i\in \N$. Moreover, there exists an equivalent code $\C^y$ such that $K_{\b 0}^y=\Diag_n(H)\leq\Aut(\C^y)$, where $y\in\Aut(\Gamma)_{\b 0}$.
\end{lemma}

\begin{proof}
 Let $h=(h_1,\ldots,h_n)\in K_{\b 0}$. If $q=2$ then, since each $h_i$ fixes $0$ and thus also fixes $1\in \Q$, it follows that $h=1$, $K_{\b 0}=1$, and the conclusion holds with $H=1$. Assume $q\geq 3$ and $K_{\b 0}\neq 1$.

 By Proposition~\ref{prop:HammingiHomogeneous}, $G_{\b 0}$ acts transitively on $\N$. Thus $K_{\b 0}^{\Q_i^\times}\cong K_{\b 0}^{\Q_j^\times}$ for all distinct $i,j\in \N$. Let $h=(h_1,\ldots,h_n)\in K_{\b 0}$, with $h\neq 1$. Let $a,a'\in\Q_i^\times$, with $a\neq a'$, $b\in\Q_j^\times$ and let $i$ and $j$ be distinct elements of $\N$. By Lemma~\ref{lem:design}, the weight $\delta$ codewords of $\C$ form a $q$-ary $2$-$(n,\delta,\lambda)$ design. Hence there exists an $\alpha,\beta\in \C$ of weight $\delta$ with $\alpha_i=a$, $\beta_i=a'$ and $\alpha_j=\beta_j=b$. Suppose that $a^{h_i}=a$. Since $K_{\b 0}$ acts trivially on $\N$ and fixes $\b 0$, we have that $\supp(\alpha)=\supp(\alpha^h)$ and $\supp(\beta)=\supp(\beta^h)$. Together with the fact that $a^{h_i}=a$, this implies that $d(\alpha,\alpha^h)< \delta$. Since $\alpha^h\in\C$ we have that $\alpha^h=\alpha$, and hence $b^{h_j}=b$. Thus we also have that $d(\beta,\beta^h)< \delta$ and $\beta^h=\beta$. Hence $(a')^{h_i}=a'$. As $a'\in\Q_i^\times\setminus\{a\}$, $b\in\Q_j^\times$ and $j\in\N\setminus\{i\}$ were chosen arbitrarily we deduce that $h_k=1$ for all $k\in\N$. Thus $K_{\b 0}\cong \Diag_n(H)$ where $H\cong K_{\b 0}^{\Q_i^\times}$. Moreover, we have shown that $K_{{\b 0},a}^{\Q_i^\times}=1$, that is, $K_{{\b 0}}^{\Q_i^\times}$ acts semi-regularly on $\Q_i^\times$, proving the first claim.
 
 Now $K_{\b 0}=\{(h,h^{\tau_2},\ldots,h^{\tau_n})\mid h\in H\}$, where $\tau_i\in \Aut(H)$ for $i=2,\ldots,n$. Let $r$ be the number of orbits of $H$ on $\Q_i^\times$. Then we can identify $\Q_i^\times$ with the disjoint union of $r$ copies of $H$, and thereby identify each $\tau_i$ with an element of $\Sym(\Q_i^\times)$. Let $y=(1,\tau_2^{-1},\ldots,\tau_n^{-1})$. Then $y\in \prod_{i\in \N}\Sym(\Q_i^\times)\leq (\Aut(\Gamma))_{\b 0}$ and,
 \begin{align*}
 (h,h^{\tau_2},\ldots,h^{\tau_n})^y= & (h, h^{\tau_2\tau_2^{-1}},\ldots, h^{\tau_n\tau_n^{-1}})\\
 = & (h,h,\ldots,h)
 \end{align*}
 Hence $\Diag_n(H)\leq(\Aut(\C^y))_{\b 0}$, completing the proof.
\end{proof}

Next we analyse, given a $(G,2)$-neighbour-transitive code $\C$ with minimum distance at least $5$, various actions of the stabilisers $G_{{\b 0},i}$ and $G_{{\b 0},i,j}$, where ${\b 0}\in\C$ and $i$ and $j$ are a pair of distinct entries in $\N$.

\begin{lemma}\label{lem:xtpitrans}
 Let $\C$ be a $(G,2)$-neighbour-transitive code with $\delta\geq 5$ in $H(\N,\Q)$, let $K$ be the kernel of the action of $G$ on $\N$, let ${\b 0}\in C$, and let $i$ and $j$ be distinct elements of $\N$. The following statements hold.
 \begin{enumerate}[$(1)$]
  \item $\Gtp$ acts transitively on each of the sets $\Q_i^\times$ and $\Q_j^\times$.
  \item $\Gtp$ has at most two orbits on $\QQ$ and if there are two orbits on $\QQ$ then they have equal size.
  \item $\Gpi/\Ki$ is isomorphic to a quotient of $G_{{\b 0},i}^\N$.
  \item $\Gtpi/\Ki$ and $G_{{\b 0},i,j}^{\QQ}/\Kij$ are isomorphic to quotients of $G_{{\b 0},i,j}^\N$.
  \item $(q-1)^2$ divides each of $2|\Gtp|$, $2|K_{\b 0}||G_{{\b 0},i,j}^\N|$ and $2|G_{{\b 0},i,j}^{\QQ}|$.
  \item $q-1$ divides $2|G_{{\b 0},i,j}^\N|$.
  \item $\left|\Gpi:\Gtpi\right|$ divides $n-1$.
 \end{enumerate}
\end{lemma}

\begin{proof}
 Since $\delta\geq 5$, we have that $G_{\b 0}$ acts transitively on $\Gamma_2(\b 0)$. Thus, the stabiliser $G_{{\b 0},\{i,j\}}$ of the subset $\{i,j\}\subseteq \N$ is transitive on the set of weight $2$ vertices with support $\{i,j\}$. Hence $\Gtp$ has at most two orbits on $\QQ$ and if there are two they have equal size. This proves part (2). By Proposition~\ref{prop:HammingiHomogeneous}, $G_{\b 0}$ acts $2$-homogeneously on $\N$. Suppose $G_{\b 0}$ is $2$-homogeneous, but not $2$-transitive, on $\N$. It follows that $\Gtp=G_{{\b 0},\{i,j\}}$ has one orbit on $\QQ$ and is thus transitive on $\Q_i^\times$ and $\Q_j^\times$. Suppose $G_{\b 0}$ is $2$-transitive on $\N$ and $\Gtp$ has two orbits on $\QQ$. Since $G_{\b 0}$ is $2$-transitive on $\N$ it follows that $\Gtpi$ is permutation isomorphic to $\Gtpj$ and hence $\Gtp$ has the same number of orbits, say $k$, on each of $\Q_i^\times$ and $\Q_j^\times$. Since each orbit of $\Gtp$ on $\QQ$ is contained in the Cartesian product of an orbit on $\Q_i^\times$ with an orbit on $\Q_j^\times$, it follows that $\Gtp$ has at least $k^2$ orbits on $\QQ$ which implies $k=1$, since if $k\geq 2$ then $k^2\geq 4$, a contradiction. Thus part (1) holds. 

 Recall that round brackets in the subscript of a group indicate that we are fixing a set point-wise. To obtain part (3), let $Y=\Gp$, $H=\Gpkeri$ and $\Omega=\Q_i^\times$ in the following. To obtain part (4), let 
 \[
  (Y,H,\Omega)=(\Gtp,\Gtpkeri,\Q_i^\times)\quad\text{and}\quad (\Gtp,\Gtpker,\QQ),
 \]
 respectively. In each case, $H$ is the kernel of the action of $Y$ on $\Omega$. Lemma~\ref{lem:semireg} implies that $H\cap K_{\b 0}=1$, since each $H$ fixes an element of $\Q_i^\times$. Thus, $K_{\b 0}= K_{\b 0}/(H\cap K_{\b 0})\cong K_{\b 0}H/H$ and $K_{\b 0}H/K_{\b 0}\cong H/(H\cap K_{\b 0})=H$. Hence,
 \[
  \frac{Y^\Omega}{K_{\b 0}^\Omega}
  \cong\frac{Y/H}{K_{\b 0}H/H}
  \cong \frac{Y}{K_{\b 0}H}
  \cong \frac{Y/K_{\b 0}}{K_{\b 0}H/K_{\b 0}}
  \cong\frac{Y/K_{\b 0}}{H}
  \cong\frac{Y^M}{H}.
 \]
 Hence parts (3) and (4) hold.
 
 By part (2), $\Gtpta$ is either transitive, or has two equal-sized orbits, on $\QQ$. Thus $(q-1)^2$ divides $2|\Gtpta|$ and also divides $2|\Gtp|$, which is equal to $2|K_{\b 0}||G_{{\b 0},i,j}^\N|$. This gives part (5). Since $G_{{\b 0},i,j}^{\QQ}\cong\Gtp/K_{\b 0}$ and Lemma~\ref{lem:semireg} implies $|K_{\b 0}|$ divides $q-1$, we have that $q-1$ divides $2|G_{{\b 0},i,j}^{\QQ}|$, proving part (6).
 
 By Proposition~\ref{prop:HammingiHomogeneous}, $G$ acts $2$-homogeneously on $\N$ and hence $|G_{{\b 0},i}^\N:G_{{\b 0},i,j}^\N|=n-1$ or $(n-1)/2$. By parts (3) and (4), there exist $N_1\nsub G_{{\b 0},i}^\N$ and $N_2\nsub G_{{\b 0},i,j}^\N$ such that $\Gpi/\Ki\cong G_{{\b 0},i}^\N/N_1$ and $\Gtpi/\Ki\cong G_{{\b 0},i,j}^\N/N_2$. This implies that $|K_{\b 0}||N_1|=|\Gpkeri|$ and $|K_{\b 0}||N_2|=|\Gtpkeri|$. Now, $\Gtpkeri=\Gtp\cap\Gpkeri$, so that $|N_2|$ divides $|N_1|$. Let $m=|N_1|/|N_2|$. Then,
 \[
  |\Gpi:\Gtpi|
  =\frac{|\Gpi/\Ki|}{|\Gtpi/\Ki|}
  =\frac{|G_{{\b 0},i}^\N/N_1|}{|G_{{\b 0},i,j}^\N/N_2|}
  =\frac{k}{m},
 \]
 where $k=n-1$ or $(n-1)/2$.  Thus, $|\Gpi:\Gtpi|$ divides $n-1$, proving part (7).
\end{proof}





 



We now turn to the to the action of $G_i$ on $\Q_i^\times$ for a $G$-alphabet affine and $(G,2)$-neighbour-transitive code, first considering a special case.

\begin{lemma}\label{lem:BinaryTrivialK0}
 Let $\C$ be a $G$-alphabet-affine and $(G,2)$-neighbour-transitive code with $\delta\geq 5$ in $H(n,q)$, where $q=2^d$, let $K$ be the kernel of the action of $G$ on $\N$, and suppose that $K_{\b 0}=1$. Then $G_{{\b 0},i}^{\Q_i^\times}\leq \GaL_1(q)$.
\end{lemma}

\begin{proof}
 Since $\C$ is $G$-alphabet-affine, we have that $G_{{\b 0},i}^{\Q_i^\times}\leq \GL_d(2)$. Let $H=G_{{\b 0},i}^{\Q_i^\times}$. It follows from Lemma~\ref{lem:xtpitrans}(1) that $H$ is transitive on $\Q_i^\times$, and hence $H$ satisfies one of the lines of Table~\ref{tab:tranLinearSemireg}. If $H\leq \GaL_1(q)$ then the result holds. So we may assume that $H$ contains one of the following groups as a normal subgroup: $\alt_6$, $\alt_7$, $\SL_{d/k}(2^k)$, $\Sp_{d/k}(2^k)$ or $G_2(2^{d/6})'$. In order to complete the proof we will eliminate these possibilities. The general strategy will be to apply Lemma~\ref{lem:xtpitrans}(3), which, since $K_{\b 0}=1$, tells us that $H$ is a quotient of $G_{{\b 0},i}$. Note that $K_{\b 0}=1$ also implies $G_{{\b 0},i}\cong G_{{\b 0},i}^\N$. In particular, any composition factor of $H$ must be a composition factor of $G_{{\b 0},i}^\N$. Proposition~\ref{prop:Hamming2transOnQ} then allows us to determine the possibilities for the $2$-homogeneous group $G_{{\b 0}}^\N$. Note also that Lemma~\ref{lem:xtpitrans}(1) tells us that $G_{{\b 0},i,j}^{\Q_i^\times}$ is transitive on $\Q_i^\times$ and Lemma~\ref{lem:xtpitrans}(5) implies that $(q-1)^2$, which is odd, divides $|G_{{\b 0},i,j}|$. 

 First, suppose that $q=16$ and $H$ contains a normal subgroup $\alt_6\cong \PSL_2(9)\cong \Sp_4(2)'$. This implies that either $G_{\b 0}^\N$ is affine and one of $\alt_6$ or $\SL_2(9)$ is a normal subgroup of $G_{{\b 0},i}$, or $G_{\b 0}^\N$ is almost-simple and contains one of $\alt_7$ or $\PSL_3(9)$ as a normal subgroup. Here we have that $(q-1)^2=3^2\cdot 5^2$, but in each of these cases $5^2$ does not divide $|G_{{\b 0},i,j}|$.
 
 Suppose that $q=16$ and $\alt_7\nsub H$. Then either $G_{\b 0}^\N$ is affine and $\alt_7\nsub G_{{\b 0},i}$, or $G_{\b 0}^\N$ is almost-simple and contains $\alt_8$ as a normal subgroup. Again $(q-1)^2=3^2\cdot 5^2$ but $5^2$ does not divide $|G_{{\b 0},i,j}|$.
 
 Suppose that $\SL_{d/k}(2^k)\nsub H$ where $k$ divides $d$ and $k<d$. Then one of the following holds:
 \begin{enumerate}[(1)]
  \item $q=2^d$, $G_{\b 0}^\N$ is affine and $\SL_{d/k}(2^k)\nsub G_{{\b 0},i}$.
  \item $q=2^d$, $G_{\b 0}^\N$ is almost-simple and $\PSL_{d/k+1}(2^k)\nsub G_{{\b 0}}$.
  \item $q=16$, $G_{\b 0}^\N$ is affine and one of $\SL_{2}(5)$ or 
 $\alt_8$ is a normal subgroup of $G_{{\b 0},i}$.
  \item $q=16$, $G_{\b 0}^\N$ is almost-simple and one of $\PSL_{3}(5)$, $\alt_6$ or $\alt_8$ is a normal subgroup of $G_{{\b 0}}$.
  \item $q=8$, $G_{\b 0}^\N$ is affine and $\SL_{2}(7)\nsub G_{{\b 0},i}$.
  \item $q=8$, $G_{\b 0}^\N$ is almost-simple and $\PSL_{3}(7)\nsub G_{{\b 0}}$.
 \end{enumerate}
 In cases (1) and (2), Zsigmondy's theorem \cite{zsigmondy1892theorie} ensures that $(q-1)^2$ does not divide $|G_{{\b 0},i,j}|$, except possibly when $d=6$. If $d=6$ and we are in case (1) or (2), it can be seen directly that $63^2$ does not divide $|G_{{\b 0},i,j}|$. For cases (3), (4) and (5), $(q-1)^2$ does not divide $|G_{{\b 0},i,j}|$ for any possible group, leaving case (6). Let $q=8$, $n=57$ and $\PSL_{3}(7)\nsub G_{{\b 0}}$. Then $G_{{\b 0},i,j}$ is isomorphic to $\AGL_1(7)\times\AGL_1(7)$, or an index $3$ subgroup of this, and neither group has a quotient isomorphic to $\GL_1(8)$ or $\GaL_1(8)$. Therefore $G_{{\b 0},i,j}$ does not act transitively on $\Q_i^\times$.
 
 
 Let $\Sp_{d/k}(2^k)\nsub H$, where $d/k$ is even and at least $4$. Then $G_{\b 0}^\N$ is affine and $\Sp_{d/k}(2^k)\nsub G_{{\b 0},i}$. We then have, by Zsigmondy's theorem \cite{zsigmondy1892theorie}, that $(2^d-1)^2$ does not divide $|G_{{\b 0},i,j}|$, except possibly when $d=6$. However, $d=6$ implies $k=1$ and in this case the order of $G_{{\b 0},i,j}=2^{1+4}:\Sp_4(2)$ is not divisible by $63^2$.
 
 If $G_2(2^{d/6})'\nsub H$ then $G_{\b 0}^\N$ is affine and $\Sp_{d/k}(2^k)\nsub G_{{\b 0},i}$. Again, Zsigmondy's theorem \cite{zsigmondy1892theorie} ensures that $(q-1)^2$ does not divide $|G_{{\b 0},i,j}|$, except in the case $d=6$. When $d=6$ we have that $|G_2(2)|=2^5\cdot 3^3\cdot 7$ is not divisible by $7^2$, ruling this case out. This completes the proof.
\end{proof}

\noindent
\begin{table}[ht]
\begin{center}
 \begin{tabular}{ccc}
  $G_0$ & parameters $q=p^d$ & semi-regular $S$\\
  \hline
  $G_0 \leq \GaL_1(p^d)$ & $q=p^d$ & numerous \\
  $\SL_{d/k}(p^k)\nsub G_0$ & $(d/k,p^k)\neq (2,2),(2,3)$ & $S \leq \F_{p^k}^\times$\\
  $\Sp_{d/k}(p^k)\nsub G_0$ & $d/k$ even & $S \leq \F_{p^k}^\times$\\
  $G_2(2^k)'\nsub G_0$ & $d=6k$, $k\geq 1$ & $S \leq\F_{2^k}^\times$\\
  $\SL_2(3)\nsub G_0$ & $p=3,5,7,11,23$, $d=2$ & $S \leq\F_p^\times$, or $p=3$ and $S={\rm Q}_8$\\
  $2_-^{1+4}\nsub G_0$ & $p^d=3^4$ & $S\leq\F_3^\times$\\
  $\SL_2(5)\nsub G_0$ & $p=11,19,29,59$, $d=2$ & $\SL_2(5)\nsub S$ or $S\leq\F_p^\times$\\
  $\alt_6\nsub G_0$ & $p=2,d=4$ & trivial\\
  $\alt_7\nsub G_0$ & $p=2,d=4$ & trivial\\
  $\SL_2(5)\nsub G_0$ & $p=3,d=4$ & $S\leq\F_9^\times$\\
  $2_-^{1+4}.\alt_5\nsub G_0$ & $p^d=3^4$ & $S\leq\F_3^\times$\\
  $\SL_2(13)\nsub G_0$ & $p=3,d=6$ & $S\leq\F_3^\times$\\
  \hline
 \end{tabular}
 \caption{Transitive linear groups $G_0\leq \GL_d(p)$ and their semi-regular normal subgroups $S\nsub G_0$. Note that in the third from last line $G_0\leq \GaL_2(9)$ and $\SL_2(5)$ is not semi-regular.}
 \label{tab:tranLinearSemireg}
\end{center}
\end{table}

The next result shows that, for a $2$-neighbour-transitive code with minimum distance at least $5$ in $H(n,q)$, the action of the stabiliser of a codeword on the alphabet is a $1$-dimensional semi-linear group.

\begin{lemma}\label{lem:entryStab1dimLinear}
 Let $\C$ be a $G$-alphabet-affine and $(G,2)$-neighbour-transitive code with $\delta\geq 5$ in $H(\N,\Q)$. Then $G_{{\b 0},i}^{\Q_i^\times}\leq \GaL_1(q)$.
\end{lemma}

\begin{proof}
 Let $q=p^d$, where $p$ is prime, let $\S$ be the subcode of $\C$ as in Proposition~\ref{orbitof0is2ntmodule}, and let $X=\Aut(\S)$. By Proposition~\ref{orbitof0is2ntmodule}(1), we have $G_{\b 0}\leq X_{\b 0}$, and hence $G_{{\b 0},i}^{\Q_i^\times}\leq X_{{\b 0},i}^{\Q_i^\times}$. Thus, it suffices to prove that $X_{{\b 0},i}^{\Q_i^\times}\leq \GaL_1(q)$. Let $L=X_{(\N)}$ be the kernel of the action of $X$ on $\N$. Note that by Proposition~\ref{orbitof0is2ntmodule}(3), $\S$ is $X$-alphabet-affine and $(X,2)$-neighbour-transitive code with minimum distance at least $5$. Moreover, by Proposition~\ref{orbitof0is2ntmodule}(3), $\S$ is an $\F_p G_{\b 0}$-module and hence $\Diag_n(\F_p^\times)\leq L_{\b 0}$.
 
 By Lemma~\ref{lem:xtpitrans}(1), $X_{{\b 0},i}^{\Q_i^\times}$, which is a subgroup of $\GL_d(p)$, is transitive on $\Q_i^\times$. By Lemma~\ref{lem:semireg}, $L_{\b 0}\cong \Diag_n(H)$, where $H$ acts semi-regularly on $\Q_i^\times$, and we may assume that $D=\{(h,h)\mid h\in H\}$. Note that $L_{\b 0}\nsub X_{\b 0}$ implies $D\nsub X_{{\b 0},i,j}^{\Q_i^\times\times \Q_j^\times}$. Hence, $X_{{\b 0},i,j}^{\Q_i^\times\times \Q_j^\times}$ is contained in the normaliser $N_{A\times A}(D)$, where $A=N_{\GL_d(p)}(H)$. Now $(h_1,h_2)\in N_{A\times A}(D)$ implies $h_1h_2^{-1}$ is an element of the centraliser $C_A(H)$ so that $h_1\in N_A(H)$ and $h_2=h_1h'$ for some $h'\in C_A(H)$. Hence, $X_{{\b 0},i,j}^{\Q_i^\times\times \Q_j^\times}$ has order dividing $|A|\cdot |C_A(H)|$. Table~\ref{tab:tranLinearSemireg} lists the transitive subgroups of $\GL_d(d)$, as well as their semi-regular normal subgroups. Recalling that $\F_p^\times\leq H$ (see the first paragraph of this proof), there are four cases we need to consider for $A$ and $H$:
 \begin{enumerate}[(1)]
  \item $H=1$, $p=2$ and $A=\GL_{d}(2)$. Here $C_A(H)=A$.
  \item There exists some $k$ dividing $d$ such that $\F_{p^k}^\times \leq H$ and $A=\GaL_{d/k}(p^k)$, where $k\geq 2$ if $p=2$.
  \item $H={\rm Q}_8$, $p=3$, $d=2$ and $A=\GL_2(3)$.
  \item $\SL_2(5)\leq H$ and $A= \F_p^\times\circ \SL_2(5)$, where $p=11,19,29$ or $59$, and $d=2$.
 \end{enumerate}
 The result holds in case (1), by Lemma~\ref{lem:BinaryTrivialK0}. If case (2) holds with $k=d$, then $X_{{\b 0},i}^{\Q_i^\times}\leq \GaL_1(q)$ and the result holds. Suppose that case (2) or (3) holds, with $k<d$ in case (2) and set $k=1$ in case (3). In each case we have $C_A(H)=\F_{p^k}^\times$. This implies that $N_{A\times A}(D)$ has orbits $\{(a,\lambda a)\mid a\in\F_{p^k}^{d/k}\setminus\{(0,0)\},\lambda\in\F_{p^k}^\times\}$ and $\{(a,b)\mid a\in\F_{p^k}^{d/k}\setminus\{(0,0)\},b\in\F_{p^k}^{d/k}\setminus\langle a\rangle\}$ on $\QQ$ of sizes $(p^d-1)(p^k-1)$ and $(p^d-1)(p^d-p^k)$, respectively. The condition $k<d$ ensures these orbits are not the same size. Since $X_{{\b 0},i,j}^{\Q_i^\times\times \Q_j^\times}$ either has precisely the same orbits, or greater number of orbits, on $\QQ$, Lemma~\ref{lem:xtpitrans}(2) rules out these cases. In case (4), $C_A(H)=\F_p^\times$ and $|A|\cdot|C_A(H)|=(p-1)^2\cdot|\SL_2(5)|/2=60(p-1)^2$. Lemma~\ref{lem:xtpitrans}(5) then implies that $(p+1)^2$ divides $120$, which is not the case for $p=11,19,29$ or $59$, so that case (4) does not hold. This completes the proof.
\end{proof}

 
 

We are now in a position to prove Theorem~\ref{thm:k0solubleandsemireg}.

\begin{proof}[Proof of Theorem~\ref{thm:k0solubleandsemireg}]
 If $\C$ is $G$-entry-faithful then, since the result holds when $q=2$, this case follows from \cite[Theorem~1.1]{ef2nt}. By \cite[Theorem~1.1]{gillespie2017alphabet}, there are no $G$-alphabet-almost-simple and $(G,2)$-neighbour-transitive codes with $\delta\geq 5$. Hence we may assume that $\C$ is $G$-alphabet-affine. Lemma~\ref{lem:entryStab1dimLinear} and Proposition~\ref{prop:Hamming2transOnQ} show that $G_i^{\Q_i}$ is a $2$-transitive subgroup of $\AGaL_1(q)$ and Lemma~\ref{lem:semireg} shows that $K_{\b 0}\cong \Diag_n(H)$, where $H$ acts semi-regularly on $\Q_i^\times$.
\end{proof}




\section{Polynomial evaluation codes}\label{sec:polyCodes}

This section presents the proof of Theorem~\ref{thm:2ntCodes}, beginning with a description of the codes involved. The codes in Theorem~\ref{thm:2ntCodes} are related to submodules of permutation modules of certain classical groups. Specifically, thinking of the vertices of $H(\N,\F_q)$ as functions $\N\rightarrow\F_q$ allows us to view a set of appropriate polynomials as a code. Historically, polynomial algebras have been used to construct many interesting examples of codes, such as the generalised Reed--Muller codes and the projective Reed--Muller codes; see Definitions~\ref{def:ReedMuller}. 

The next result develops a concrete connection between $H(\N,\F_q)$ and $R(q,s,t,k)$.

\begin{lemma}\label{lem:polyHammingBij}
 Let $R=R(q,s,t,k)$, as in Definition~\ref{def:polynomials}, let $\N$ be a set of representatives for a subset of the set of all $1$-dimensional subspaces of $\F_{q^s}^t$, let $P=\F_{q^s}[x_1,\ldots,x_t]$, and let $I$ be the ideal of $P$ consisting of the set of all polynomials vanishing on $\bigcup_{v\in\N}\langle v\rangle_{\F_{q^s}}$. Then the following hold.
 \begin{enumerate}[$(1)$]
  \item $R$ and $R\cap I$ are $\F_q$-vector spaces.
  \item There is an $\F_q$-vector space isomorphism from $R/(R\cap I)$ to the vertex set of $H(\N,\F_q)$.
  \item Both $R$ and $R/(R\cap I)$ are $\F_q \GL_t(q^s)$-modules. In particular, $\GL_t(q^s)/\Z_{(q^s-1)/(q-1)}$ acts as a group of automorphisms of $H(\N,\F_q)$.
 \end{enumerate}
\end{lemma}

\begin{proof}
 Let $f,f'\in R$, $(a_1,\ldots,a_t)\in\F_{q^s}^t$ and $b,c\in\F_q$. Then $(bf+cf')(a_1,\ldots,a_t)=bf(a_1,\ldots,a_t)+cf'(a_1,\ldots,a_t)\in\F_q$ and every monomial of $af+bf'$ has degree a multiple of $k(q^s-1)/(q-1)$, which shows that $R$ is an $\F_q$-vector space. Since $I$ is an $\F_q$-vector space, so is $R\cup I$, and part (1) holds. 

 Let $\alpha:\N\rightarrow\F_q$ be a vertex of $H(\N,\F_q)$. By Lagrange interpolation (see \cite[Theorem~1.7.1]{lidl1997finite}), there exists a polynomial $f_\alpha\in P$ such that $f_\alpha(a_1,\ldots,a_t)=\alpha(a_1,\ldots,a_t)$ for all $(a_1,\ldots,a_t)\in\N$ and $f_\alpha(\lambda b_1,\ldots,\lambda b_t)=N(\lambda)^k f_\alpha(b_1,\ldots,b_t)$. It follows that $f_\alpha\in R$. Moreover, for all $f_0\in R\cup I$ and every $(a_1,\ldots,a_t)\in\N$ we have $f_0(a_1,\ldots,a_t)=0$, and hence the function $\alpha$ is equal to the restriction of $f_\alpha+f_0$ to $\N$. Thus the set of cosets of $R\cap I$ in $R$ are in bijection with the set of all functions $\N\rightarrow\F_q$, and part (2) holds.
 
 The natural action of $\GL_t(q^s)$ on $\F_{q^s}^t$ induces and action on $P$ via 
 \[
  f^g(x_1,\ldots,x_t)=f((x_1,\ldots,x_t)^{g^{-1}}),
 \]
 where $g\in \GL_t(q^s)$ and $f\in P$. If $g\in \GL_t(q^s)$, then for each $i\in\{1,\ldots,t\}$ we have $(x_i)^g=a_{i1}x_1+\cdots+a_{it}x_t$, where the elements $a_{ij}$ in $\F_{q^s}$, $j=1,\ldots,t$, are not all zero. It follows from this that if $m$ is a monomial in $P$, then $m^g$ is a polynomial consisting of monomials each having the same degree as $m$. Hence $\GL_t(q^s)$ leaves both $R$ and $R\cap I$ invariant. If $g$ is in the center of $\GL_t(q^s)$ then, for some $a\in\F_{q^s}^\times$, we have that $f^g(x)=f(ax)=N(a)^k f(x)$. In particular, $g$ acts trivially on $R$ when $N(a)=1$. This proves part (3).
\end{proof}

The next result proves, under appropriate circumstances, that for $G_{\b 0}$ and $\N$ as in one of the lines of Table~\ref{tab:2ntGroups} the action of $G_{\b 0}$ on $H(\N,\F_q)$, as in Lemma~\ref{lem:polyHammingBij}, satisfies the hypotheses of Proposition~\ref{prop:coveringRadius1}. Note that the Suzuki group $\Sz(q)$, where $q=2^{2f+1}$ for some positive integer $f$, acts $2$-transitively on the Suzuki--Tits ovoid, which consists of $q^2+1$ points of the projective space $\PG_3(q)$, no three of which are collinear; see \cite[p.~250]{dixon1996permutation}. The unitary group $\PGU_3(q)$ acts $2$-transitively on the unital consisting of the $q^3+1$ isotropic points of $\PG_2(q^2)$ under a non-degenerate Hermitian form; see \cite[p.~248]{dixon1996permutation}. Recall that we denote by $T_\C$ the group of translations by elements of a linear code $\C$.



\begin{proposition}\label{prop:2ntCodes}
 Let $q$, $s$, $t$, $G_{\b 0}$ and $\N$ be as in one of the lines of Table~\ref{tab:2ntGroups}, and let $k\in\{1,2,\ldots,q-1\}$ with $\gcd(k,q-1)=1$. Then, the action of $G_{\b 0}$ on $\Gamma=H(\N,\F_q)$, as in Lemma~\ref{lem:polyHammingBij}(3), is transitive on each of the sets $\Gamma_1({\b 0})$ and $\Gamma_2({\b 0})$.
\end{proposition}

\begin{proof}
 Since $G_{\b 0}\leq \GL_t(q^s)$ in each case, Lemma~\ref{lem:polyHammingBij}(3) shows that $G_{\b 0}$ acts on $\Gamma$, where the vertex set of $\Gamma$ is identified with $R/(R\cap I)$, as in Lemma~\ref{lem:polyHammingBij}(2). Since $G_{\b 0}$ acts $2$-transitively on $\N$ and $\Diag_n(\F_q^\times)$ acts transitively on the set of vertices that are non-zero in a specified entry, it suffices to show that $G_{\b 0}$ acts transitively on the set vertices that are non-zero in two fixed entries $i,j\in\N$. To this end, let $a,b\in \F_q^\times$, let $\alpha$ be the weight two vertex of $\Gamma$ such that $\alpha(i)=a$ and $\alpha(j)=b$, and let $\beta$ be the weight two vertex satisfying $\beta(i)=\beta(j)=1$. In the remainder of the proof we check that there exists $g\in G_{\b 0}$ such that $\alpha^g=\beta$ in each case from Table~\ref{tab:2ntGroups}.

 First, suppose that $G_{\b 0}=\GL_t(q^s)/\Z_{(q^s-1)/(q-1)}$ and $\N$ is a set of representatives for the set of all $1$-dimensional $\F_{q^s}$-subspaces of $\F_{q^s}^t$. If $e_1,\ldots,e_t$ is an $\F_q^s$-basis for $\F_{q^s}^t$ then we may assume that $i=e_1$ and $j=e_2$. Since $\gcd(k,q-1)=1$, there exists $c,d\in\F_{q^s}^\times$ such that $N(c)^k=a$ and $N(d)^k=b$. Let $g\in\GL_t(q^s)$ with $e_1^g=ce_1$, $e_2^g=de_2$ and $e_\ell^g=e_\ell$ for $\ell=3,\ldots,t$. Then $\alpha^g(i)=\alpha(c^{-1}i)=N(c)^{-k}\alpha(i)=1$ and $\alpha^g(j)=\alpha(d^{-1}j)=N(d)^{-k}\alpha(j)=1$ so that $\alpha^g=\beta$ as required.

 Next, let $G_{\b 0}=\F_q^\times\rtimes \AGL_{t-1}(q^s)$, let $e_1,\ldots,e_t$ be an $\F_{q^s}$-basis for $\F_{q^s}^t$, and let $\N=\{e_1+a_2e_2+\ldots+a_te_t\mid a_2,\ldots,a_t\in\F_q^s\}$ be a set of representatives for the points of the affine geometry $\AG_{t-1}(q^s)$. Let $i=e_1$, $j=e_1+e_2$. Again, $\gcd(k,q-1)=1$ implies that there exists $c,d\in \F_{q^s}^\times$ such that $N(c)^k=a$ and $N(d)^k=b$ and so there exists some $g\in\GL_{t+1}(q^s)$ such that $e_1^g=ce_1$, $(e_1+e_2)^g=d(e_1+e_2)$ and $e_\ell^g=e_\ell$ for $\ell= 3,\ldots,t$. Since $g$ fixes $\langle e_1\rangle$, $g$ also fixes $\AG_{t-1}(q^s)$ and so has an induced action on $H(\N,\F_q)$. Again, we have $\alpha^g(i)=\alpha(c^{-1}i)=N(c)^{-k}\alpha(i)=1$ and $\alpha^g(j)=\alpha(d^{-1}j)=N(d)^{-k}\alpha(j)=1$, and hence $G_{\b 0}$ acts transitively on $\Gamma_2({\b 0})$.

 Consider now $G_{\b 0}=\F_q^\times\rtimes \PGU_3(q^{s/2})$, noting that $q$ is even here, and let $\N$ be a set of representatives for the points of a classical unital in $\pg_2(q^s)$. Let $e_1,e_2,e_3$ be a basis for $\F_{q^s}^3$ and let $i=e_1$, $j=e_3$. Following \cite[Section~7.7]{dixon1996permutation}, we have that the stabiliser of $\langle i\rangle$ and $\langle j \rangle$ in $G_{\b 0}$ contains the $3\times 3$ diagonal matrix ${\rm diag}(\lambda,1,\lambda^{-q^{s/2}})$, for any $\lambda\in\F_{q^{s}}^\times$. Note also that the classical unital is stabilised by the diagonal matrix $\mu I$, for $\mu\in\F_{q^{s}}^\times$. Taking the product of $\lambda\mu I$ and ${\rm diag}(\lambda^{-1},1,\lambda^{q^{s/2}})$ we have that there exists $g\in G_{\b 0}$ where $g$ can be be represented by the matrix ${\rm diag}(\mu,\lambda\mu,\lambda^{q^{s/2}+1}\mu)$. Since $q$ is even, we have that $\gcd(k,q-1)=\gcd(2k,q-1)=1$ and so we may choose $\lambda\in\F_{q^s}^\times$ and $\mu\in\F_{q^s}\times$ so that $N(\lambda)^{2k}=a^{-1}b$ and $N(\mu)^k=a$. Thus $\alpha^g(i)=\alpha(\mu^{-1}i)=N(\mu)^{-k}\alpha(i)=1$ and $\alpha^g(j)=\alpha((\lambda^{q^{s/2}+1}\mu)^{-1}j)=N(\lambda)^{-k(q^{s/2}+1}N(\mu)^{-k}\alpha(j)=N(\lambda)^{-2k}N(\mu)^{-k}\alpha(j)=1$, and hence $\alpha^g=\beta$.

 Next, let $G_{\b 0}=\F_q^\times\Sz(q^s)$ and let $\N$ be a set of representatives for the points of the Suzuki--Tits ovoid in $\pg_3(q^s)$. Following \cite[Section~4.2.2]{wilson2009finite}, elements of $\Sz(q^s)$ are $4\times 4$ matrices with respect to the ordered basis $\{e_1,e_2,f_2,f_1\}$. Let $i=e_1$ and $j=f_1$. Then the stabiliser of $\langle i\rangle$ and $\langle j\rangle$ contains the diagonal matrix ${\rm diag}(c,c^{2^{n+1}-1},c^{-2^{n+1}+1},c^{-1})$
 for any $c\in\F_{q^s}^\times$. Note also that $cI$ stabilises the Suzuki--Tits ovoid for any $c\in\F_{q^s}^\times$. Since $q$ is even, we have that $\gcd(k,q-1)=\gcd(2k,q-1)=1$ and hence we may choose $c,d\in\F_{q^s}^\times$ so that $N(c)^{2k}=ab^{-1}$, $N(d)^k=b$ and then let ${\rm diag}(c^2 d,c^{2^{n+1}}d,c^{-2^{n+1}+2}d,d)$ represent an element $g\in G_{\b 0}$. It then follows that $\alpha^g(i)=\alpha(c^{-2}d^{-1}i)=N(c)^{-2k}N(d)^{-k}\alpha(i)=1$ and $\alpha^g(j)=\alpha(d^{-1}j)=N(d)^{-k}\alpha(j)=1$, and indeed $\alpha^g=\beta$. Thus the result holds.
\end{proof}

\begin{proof}[Proof of Theorem~\ref{thm:2ntCodes}]
 Since $\C$ is $\F_q$-linear, the group $T_\C$ of translations by elements of $\C$ acts transitively on $\C$. Moreover, by Proposition~\ref{prop:2ntCodes}, $G_{\b 0}$ acts transitively on each of the sets $\Gamma_1({\b 0})$ and $\Gamma_2({\b 0})$. Applying Proposition~\ref{prop:coveringRadius1} and Remark~\ref{rem:perfectHamming} completes the proof.
\end{proof}

The next result shows that the conclusion of Proposition~\ref{prop:2ntCodes} is false if the condition $\gcd(k,q-1)=1$ is dropped.

\begin{proposition}
 Let $\N$ be a set of representatives for a subset of the set of all $1$-dimensional subspaces of $\F_{q^s}^t$, let $k\in\{1,\ldots,q-1\}$ such that $\gcd(k,q-1)>1$, and let $G_{\b 0}\leq \Diag_\N(\F_q^\times)\circ\GL_t(q^s)$, where $\GL_t(q^s)$ acts on $\Gamma=H(\N,\F_q)$ as in Lemma~\ref{lem:polyHammingBij}(3). Then $G_{\b 0}$ is not transitive on $\Gamma_2({\b 0})$.
\end{proposition}

\begin{proof}
 Let $\ell=\gcd(k,q-1)$. For any $u,v\in\N$, the subgroup induced by $\Diag_n(\F_q^\times)$ and the stabiliser of $\langle u\rangle\cup\langle v\rangle$ inside $\GL_t(q^s)$ on the projection $H(\{u,v\},\F_q)$ is the group
 \[
  H=\{(\omega^a,\omega^b)\sigma^c\mid a-b\equiv 0 {\pmod{\ell}}\},
 \]
 where $\omega$ is a generator for $\F_q^\times$ and $\sigma=(u\, v)$. Since $\ell>1$, it follows that there is no element $h\in H$ such that $(1,1)^h=(1,\omega)$, and hence $G_{\b 0}$ is not transitive on $\Gamma_2({\b 0})$.
\end{proof}

The next result determines the minimum distances of certain codes satisfying Theorem~\ref{thm:2ntCodes}. In particular, this shows that there are infinitely many $2$-neighbour-transitive codes for each of the lines of Table~\ref{tab:2ntGroups}.

\begin{proposition}\label{prop:2ntCodesMinDist}
 Let $q$, $s$, $t$, $G_{\b 0}$, $\N$ be as in one of the lines of Table~\ref{tab:2ntGroups}, let $\C$ be the $\F_q G_{\b 0}$-submodule of $R(q,s,t,1)$ consisting of those polynomials of degree $(q^s-1)/(q-1)$, and let $\delta$ be the minimum distance of $\C$. Then the following hold.
 \begin{enumerate}[$(1)$]
  \item If $s=1$ and $G_{\b 0}=\F_q^\times\rtimes\AGL_{t-1}(q)$ then $\delta=q^{t-1}-q^{t-2}$.
  \item If $s=2$ and $G_{\b 0}=\F_q^\times\rtimes\PGU_3(q)$ then $\delta=q^3-2q$.
  \item If $s=1$ and $G_{\b 0}=\F_q^\times\rtimes\Sz(q)$ then $\delta=q^2-q$.
 \end{enumerate}
\end{proposition}

\begin{proof}
 Since each code is linear, the minimum distance is equal to the minimum weight of a non-zero codeword. First, suppose $s=1$ and $G_{\b 0}=\F_q^\times\rtimes\AGL_{t-1}(q)$. Then $\C$ consists of degree $1$ polynomials, which are zero precisely on a codimension $1$ subspace of $\F_q^{t}$, and so are zero on either no elements of $\N$, or on 
 \[
  \frac{q^{t-1}-1}{q-1}-\frac{q^{t-2}-1}{q-1}
 \]
 elements of $\N$, since $\N$ is a set of representatives for the complement of a hyperplane in $\pg_{t-1}(q)$. Thus $\C$ has minimum distance $q^{t-1}-q^{t-2}$. 

 Next, suppose $s=2$ and $G_{\b 0}=\F_q^\times\rtimes\PGU_3(q)$. Then a polynomial $f$ of degree $q+1$ in $R(q,2,3,1)$ is zero precisely on a Baer subplane of $\PG_2(q^2)$, that is, $f$ is linear upon reduction to the subfield $\F_q$. By \cite[Corollary~8]{Barwick2000}, each Baer subplane meets a unital in either $1$, $q+1$, or $2q+1$ points. Thus, since $n=q^3+1$, the minimum distance of $\C$ is $q^3-2q$. 

 Finally, suppose $s=1$ and $G_{\b 0}=\F_q^\times\rtimes\Sz(q)$. A degree $1$ polynomial $f$ in $R(q,1,4,1)$ will evaluate to zero on precisely the points of a hyperplane in $\PG_3(q)$. Any hyperplane is either tangent to an ovoid, or meets the ovoid in an `oval' consisting of $q+1$ points. It follows that $f$ is non-zero on either $q^2$ or $q^2-q$ points of the ovoid. Thus, the minimum distance of $\C$ is $q^2-q$. 
\end{proof}

\section{Relation to Reed--Muller codes}\label{sec:ReedMuller}

The generalised Reed--Muller codes, projective Reed--Muller codes, and their subfield subcodes are defined below. The generalised Reed--Muller codes were introduced in \cite{kasami1968new} and \cite{weldon1968new}; see also \cite{delsarte1970generalized}. The projective Reed--Muller codes were first studied in \cite{vladut1984linear} and \cite{lachaud1988projective}. Assmus and Key \cite[Section~5.7]{assmus1994designs} construct and analyse the \emph{subfield subcodes} of the generalised and projective Reed--Muller codes. The parameters of the generalised Reed--Muller codes are given in \cite[Theorem~5.4.1 and Corollary~5.5.4]{assmus1994designs}, the minimum distance of the projective Reed--Muller codes can be found in \cite[Theorem~1]{sorensen1991projective}. 


\begin{definition}\label{def:ReedMuller}
 Let $q$ be a prime power, let $s$ be a positive integer, let $k$ be an integer satisfying $0<k\leq q-1$, let $\ell$ be a multiple of $k(q^s-1)/(q-1)$, let $P=\F_q[x_1,\ldots,x_t]$, and let $R(q,s,t,k)$ be as in Definition~\ref{def:polynomials}. Define the following codes.
 \begin{enumerate}[(1)]
  \item Set $k=s=1$. The \emph{generalised Reed--Muller code} $\RM_q(\ell,t)$ in $H(\F_q^t,\F_q)$ is given by the subspace of $P$ consisting of all polynomials of degree at most $\ell$. 
  \item Set $k=1$. The \emph{subfield subcode} $\RM_{q^s/q}(\ell,t)$, of $\RM_{q^s}(\ell,t)$, in $H(\F_{q^s}^t,\F_q)$ is given by the $\F_q$-subspace consisting of all $f\in\RM_{q^s}(\ell,t)$ such that $f(v)\in\F_q$ for all $v\in\F_{q^s}^t$.
  \item Let $s=1$ and $\N$ be a set of representatives for the set of all $1$-dimensional subspaces of $\F_q^t$. The \emph{projective Reed--Muller} code $\PRM_q(\ell,t)$ in $H(\N,\F_q)$ is given by the subspace of $R(q,1,t,k)$ consisting of all polynomials of degree at most $\ell$. 
  \item Let $\N$ be a set of representatives for the set of all $1$-dimensional subspaces of $\F_{q^s}^t$. The \emph{subfield subcode} $\PRM_{q^s/q}(\ell,t)$, of $\PRM_{q^s}(\ell,t)$, in $H(\N,\F_q)$ is given by the subspace of $R(q,s,t,k)$ consisting of all polynomials of degree at most $\ell$. 
 \end{enumerate}
\end{definition}

 It is worth briefly comparing the automorphism groups of the codes $\RM_{q^s}(\ell,t-1)$ and $\RM_{q^s/q}(\ell,t-1)$ with the codes of Theorem~\ref{thm:2ntCodes} where $G_{\b 0}$ and $\N$ are as in line 2 of Table~\ref{tab:2ntGroups}. In each case the action of the automorphism group of the code induces a faithful action of $\AGL_{t-1}(q^s)$ on $\N$. However, while the codes $\RM_{q^s}(\ell,t-1)$ and $\RM_{q^s/q}(\ell,t-1)$ are preserved by $\AGL_{t-1}(q^s)$ as a subgroup of $\Sym(\N)$, this is not the case for the codes arising from line 2 of Table~\ref{tab:2ntGroups}. generally, if $g=h\sigma\in\AGL_{t-1}(q^s)\leq G_{\b 0}$, where $h\in\Sym(\Q)^n$ and $\sigma\in\Sym(\N)$, then $\sigma\neq 1$ implies $h\neq 1$. 
 
 The next result concerns the $2$-neighbour-transitivity of the generalised Reed--Muller codes, the projective Reed--Muller codes, and their subfield subcodes. Note that if $\ell=(t-1)(q-1)-1$ then $\PRM_q(\ell,t)$ is a perfect Hamming code and is neighbour-transitive, but not $2$-neighbour-transitive; see Remark~\ref{rem:perfectHamming}. 

\begin{proposition}\label{prop:2ntReedMuller}
 Following the notation as in Definition~\ref{def:ReedMuller}, let one of the following hold:
 \begin{enumerate}[$(1)$]
  \item $\C=\RM_{2}(\ell,t)$ with $\ell\leq t-2$.
  \item $\C=\RM_{2^s/2}(\ell,t)$ with $\ell\leq (t-1)(2^s-1)-1$.
  \item $\C=\PRM_{q}(\ell,t)$ with $\ell\leq (t-1)(q-1)-2$ and $\gcd(k,q-1)=1$.
  \item $\C=\PRM_{q^s/q}(\ell,t)$ with $\ell\leq (t-1)(q^s-1)-2$ and $\gcd(k,q-1)=1$.
 \end{enumerate}
 Then $\C$ is $2$-neighbour-transitive. 
\end{proposition}

\begin{proof}
 Since $\C$ is linear in all cases, the group of translations $T_\C$ acts transitively on $\C$. In (1) and (2), $\Aut(\C)_{\b 0}$ contains $\AGL_t(2)$ and $\AGL_t(2^s)$, respectively, so that $\Aut(\C)_{\b 0}$ acts transitively on each of the sets $\Gamma_1({\b 0})$ and $\Gamma_2({\b 0})$. For cases (3) and (4), Proposition~\ref{prop:2ntCodes} implies that $\Aut(\C)_{\b 0}$ acts transitively on each of $\Gamma_1({\b 0})$ and $\Gamma_2({\b 0})$. We claim that each code has minimum distance at least $4$, in which case the result follows from Proposition~\ref{prop:coveringRadius1}. To prove the claim, first observe that the minimum distance of $\RM_{2^s/2}(\ell,t)$ is bounded below by the minimum distance of $\RM_{2^s}(\ell,t)$, and the minimum distance of $\PRM_{q^s/q}(\ell,t)$ is bounded below by $\PRM_{q^s}(\ell,t)$. By \cite[Corollary~5.5.4]{assmus1994designs}, 
 $\RM_{2^s}(\ell,t)$ has minimum distance $(2^s-m)2^{s(t-r-1)}$, where $\ell=r(2^s-1)+m$ with $m\in\{0,\ldots,2^s-1\}$. Thus the result holds in cases (1) and (2). By \cite[Theorem~1]{sorensen1991projective}, 
 $\PRM_{q^s}(\ell,t)$ has minimum distance $(q^s-m)q^{s(t-r-1)}$, where $\ell-1=r(q^s-1)+m$ with $m\in\{0,\ldots,2^s-1\}$. Thus the result also holds in cases (3) and (4).
\end{proof}

\end{document}